\input xy
\xyoption{all}\xyoption{poly}
\newdir{ >}{{}*!/-9pt/@{>}}

\documentclass[reqno]{amsart}
\usepackage{amsmath,latexsym,amssymb,mathrsfs}

\usepackage{textcomp}
\usepackage{graphicx}

\usepackage{hyperref}
\usepackage[all]{xy}

\newcommand{\Ker}{\mathsf{Ker}}

\newcommand{\Eq}{\mathsf{Eq}}

\theoremstyle{plain}
\newtheorem{theorem}{Theorem}[section]
\newtheorem{lemma}[theorem]{Lemma}
\newtheorem{proposition}[theorem]{Proposition}
\newtheorem{corollary}[theorem]{Corollary}

\theoremstyle{definition}
\newtheorem{definition}[theorem]{Definition}

\theoremstyle{remark}

\newdir{ >}{% end of arrow for the monos
  @{}*!/-10pt/@{>} }

\newdir{ |>}{% end of arrow for the kernels
  @{}*!/-5.5pt/@{|}*!/-10pt/:(1,-.2)@^{>}*!/-10pt/:(1,+.2)@_{>} }

\newcommand{\CC}{ \ensuremath{\mathbb {C}} }
\newcommand{\N}{\mathcal N}

\newcommand{\Nk}{ \ensuremath{\mathrm {N}} }
\newcommand{\n}{ \ensuremath{\mathrm {n}} }
\newcommand{\rr}{\rightrightarrows}

\newcommand{\map}[2]{ \ensuremath{ \xymatrix@1@C=15pt{ #1 \ar[r] & #2 } } }
\newcommand{\mono}[2]{ \ensuremath{ \xymatrix@1@C=15pt{ #1 \ar@{ >->}[r] & #2 } } }
\newcommand{\regepi}[2]{ \ensuremath{ \xymatrix@1@C=15pt{ #1 \ar@{>>}[r] & #2 } } }

\newcommand{\sarl}[2]{\ar@<1.6pt>[#2]^-{#1}\ar@<-1.6pt>[#2]}
\newcommand{\sarldash}[2]{\ar@{-->}@<1.6pt>[#2]^-{#1}\ar@{-->}@<-1.6pt>[#2]}
\newcommand{\sarr}[2]{\ar@<1.6pt>[#2]\ar@<-1.6pt>[#2]_-{#1}}
\newcommand{\sarlb}[2]{\ar@/^4pt/@<1.6pt>[#2]^-{#1}\ar@/^4pt/@<-1.6pt>[#2]}
\newcommand{\sarrb}[2]{\ar@/_4pt/@<1.6pt>[#2]\ar@/_4pt/@<-1.6pt>[#2]_-{#1}}
\newcommand{\sarc}[2]{\ar@<0pt>@{}[#2]|-{#1}\ar@<1.6pt>[#2]\ar@<-1.6pt>[#2]}
\newcommand{\sarlh}[3]{\ar@/#3/@<1.6pt>[#2]^-{#1}\ar@/#3/@<-1.6pt>[#2]}
\newcommand{\sarrh}[3]{\ar@/#3/@<1.6pt>[#2]\ar@/#3/@<-1.6pt>[#2]_-{#1}}
\newcommand{\sarch}[3]{\ar@/#3/@<0pt>@{}[#2]|-{#1}\ar@/#3/@<1.6pt>[#2]\ar@<-1.6pt>[#2]}

\begin{document}
\title{ON $2$-STAR-PERMUTABILITY IN REGULAR MULTI-POINTED CATEGORIES}

\author{Marino Gran}
\address{
\noindent  Institut de Recherche en Math\'ematique et Physique, Universit\'e catholique de Louvain, Chemin du Cyclotron 2, 1348 Louvain-la-Neuve, Belgium}

\author{Diana Rodelo}    
\address{  Departamento de Matem\'atica, Faculdade de Ci\^{e}ncias e Tecnologia\\ Universidade do Algarve, Campus de
Gambelas\\  8005--139 Faro, Portugal\\and \\
 CMUC, Universidade de Coimbra\\  3001-454 Coimbra, Portugal \\ drodelo@ualg.pt
 }

\begin{abstract} 
$2$-star-permutable categories were introduced in a joint work with Z. Janelidze and A. Ursini as a common generalisation of regular Mal'tsev categories and of normal subtractive categories. In the present article we first characterise these categories in terms of what we call star-regular pushouts. We then show that the $3 \times 3$ Lemma characterising normal subtractive categories and the Cuboid Lemma characterising regular Mal'tsev categories are special instances of a more general homological lemma for star-exact sequences. We prove that $2$-star-permutability is equivalent to the validity of this lemma for a star-regular category.
\\

\end{abstract}
\maketitle

\section*{Introduction}
The theory of \emph{Mal'tsev categories} in the sense of A. Carboni, J. Lambek and M.C. Pedicchio  \cite{CLP} provides a beautiful example of the way how categorical algebra leads to a structural understanding of algebraic varieties (in the sense of universal algebra). Among regular categories, Mal'tsev categories are characterised by the property of $2$-permutability of equivalence relations: given two equivalence relations $R$ and $S$ on the same object $A$, the two relational composites $RS$ and $SR$ are equal:
$$RS =SR.$$
In the case of a variety of universal algebras this property is actually equivalent to the existence of a ternary term $p(x,y,z)$ satisfying the identities $p(x,y,y)=x$ and $p(x,x,y)=y$ \cite{S}. In the pointed context, that is when the category has a zero object, there is also a suitable notion of $2$-permutability, called ``$2$-permutability at $0$'' \cite{Ursini}. In a variety this property can be expressed by requiring that, whenever for a given element $x$ in an algebra $A$ there is an element $y$ with $xRyS0$ (here $0$ is the unique constant in $A$), then there is also an element $z$ in $A$ with $xSzR0$. The validity of this property is equivalent to the existence of a binary term $s(x,y)$ such that the identities $s(x,0)=x$ and $s(x,x)=0$ hold true \cite{Ursini}.
Among regular categories, the ones where the property of $2$-permutability at $0$ holds true are precisely the \emph{subtractive categories} introduced in \cite{ZJ1}.

The aim of this paper is to look at regular Mal'tsev and at subtractive categories as special instances of the general notion of $2$-star-permutable categories introduced in collaboration with Z. Janelidze and A. Ursini in \cite{GraJanRodUrs11}. This generalisation is achieved by working in the context of a \emph{regular multi-pointed category}, i.e. a regular category equipped with an ideal $\N$ of distinguished morphisms \cite{Ehr64}. When $\N$ is the class of all morphisms, a situation which we refer to as the \emph{total context}, regular multi-pointed categories are just regular categories, and $2$-star-permutable categories are precisely the regular Mal'tsev categories. When $\N$ is the class of all zero morphisms in a pointed category, we call this the \emph{pointed context}, regular multi-pointed categories are regular pointed categories, and $2$-star-permutable categories are the regular subtractive categories.

This paper follows the same line of research as in \cite{GraJanRodUrs11} which was mainly focused on the property of $3$-star-permutability, a generalised notion which captures Goursat categories in the total context and, again, subtractive categories in the pointed context.

 In this work we study two remarkable aspects of the property of $2$-star-permutability. First we provide a characterisation of $2$-star-permutable categories in terms of a special kind of pushouts (Proposition \ref{starregularpushoutscharacterisation}), that we call \emph{star-regular pushouts} (Definition \ref{Definition srp}). Then we examine a homological diagram lemma of star-exact sequences, which can be seen as a generalisation of the $3 \times 3$ Lemma, whose validity is equivalent to $2$-star-permutability.  We call this lemma the Star-Upper Cuboid Lemma (Theorem \ref{thm star-upper cuboid}). The validity of this lemma turns out to give at once a characterisation of regular Mal'tsev categories (extending a result in \cite{GR}) and, in the pointed context, a characterisation of those normal categories which are subtractive (this was first discovered in \cite{ZJan06}).

{\bf Acknowledgement.} The authors are grateful to Zurab Janelidze for some useful conversations on the subject of the paper.
\section{Star-regular categories}\label{Star-regular categories}

\subsection{Regular categories and relations}\label{Regular categories and relations}
A finitely complete category $\CC$ is said to be a \emph{regular} category \cite{EC} when any kernel pair has a coequaliser and, moreover,
regular epimorphisms are stable under pullbacks. In a regular category any morphism $f \colon X \rightarrow Y$ has a factorisation $f=m\cdot p$,  where $p$  is a regular epimorphism and $m$ is a monomorphism. The corresponding (regular epimorphism, monomorphism) factorisation system is then stable under pullbacks.

%\noindent \emph{In this article $\CC$ will always denote a finitely complete regular category.}

 A relation $\varrho$ from $X$ to $Y$ is a subobject $\langle \varrho_1,\varrho_2 \rangle\colon R\rightarrowtail X\times Y$. The opposite relation, denoted $\varrho^{\circ}$, is given by the subobject $\langle \varrho_2,\varrho_1 \rangle\colon R\rightarrowtail Y\times X$. We identify a morphism $f:X\rightarrow Y$ with the relation $\langle 1_X,f \rangle\colon X\rightarrowtail X\times Y$ and write $f^{\circ}$ for the opposite relation. Given another relation $\sigma$ from $Y$ to $Z$, the composite relation of $\varrho$ and $\sigma$ is a relation $\sigma \varrho$ from $X$ to $Z$. With this notation, we can write the above relation as $\varrho=\varrho_2 \varrho_1^{\circ}$. The following properties are well known (see \cite{CKP}, for instance); we collect them in a lemma for future references.

\begin{lemma}
\label{pps of ms as relations}
Let $f: X\rightarrow Y$ be any morphism in a regular category $\CC$. Then:
\begin{enumerate}
 \item {$f f^{\circ} f=f$ and $f^{\circ} f f^{\circ}=f^{\circ}$;}
  \item{%$f f^{\circ}$ is $(m,m)$, and
  $f f^{\circ}=1_Y$ { if and only if} $f$
  is a regular {epimorphism}.}
\end{enumerate}
\end{lemma}

A kernel pair of a morphism $f: X\rightarrow Y$, denoted by $$(\pi_1,\pi_2)\colon \Eq(f) \rr X,$$ is called an \emph{effective equivalence relation}; we write it either as $\Eq(f)=f^{\circ}f$, or as $\Eq(f)=\pi_2 \pi_1^{\circ}$, as mentioned above. When $f$ is a regular epimorphism, then $f$ is the coequaliser of
$\pi_1$ and $\pi_2$ and the diagram
$$
\xymatrix{
  \Eq(f) \ar@<3pt>[r]^-{\pi_1} \ar@<-3pt>[r]_-{\pi_2} & X \ar@{>>}[r]^-f & Y }
$$
is called an \emph{exact fork}. In a regular category any effective equivalence relation is the kernel pair of a regular epimorphism.

\subsection{Star relations}
We now recall some notions introduced in \cite{GJU}, which are useful to develop a unified treatment of pointed and non-pointed categorical algebra.
Let $\CC$ denote a category with finite limits, and $\N$ a distinguished class of morphisms that forms an \emph{ideal}, i.e.~for any composable pair of morphisms $g,f$, if either $g$ or $f$ belongs to $\N$, then the composite $g\cdot f$ belongs to $\N$. An \emph{$\N$-kernel} of a morphism $f:X\rightarrow Y$ is defined as a morphism $\n_{f}: \Nk_f \rightarrow X$ such that $f\cdot \n_f \in\N$ and $\n_f$ is universal with this property (note that such $\n_f$ is automatically a monomorphism). A pair of parallel morphisms, denoted by $\sigma=(\sigma_1,\sigma_2):S\rr X$ with $\sigma_1\in \N$, is called a \emph{star}; it is called a monic star, or a \emph{star relation}, when the pair $(\sigma_1,\sigma_2)$ is jointly monomorphic. %A star $\sigma=(\sigma_1,\sigma_2)$ with both $\sigma_1, \sigma_2\in \N$ is said to be a \emph{bi-star}.

Given a relation $\varrho=(\varrho_1,\varrho_2):R\rr X$ on an object $X$, we denote by $\varrho^\ast:R^\ast\rr X$ the biggest subrelation of $\varrho$ which is a (monic) star. When $\CC$ has $\N$-kernels, it can be constructed by setting $\varrho^\ast=(\,\varrho_1\cdot \n_{\rho_1},\varrho_2\cdot \n_{\rho_1}\,)$, where $\n_{\rho_1}$ is the $\N$-kernel of $\varrho_1$. In particular, if we denote the discrete equivalence relation on an object $X$ by $\Delta_X=(1_X,1_X):X\rr X$, then $\Delta_X^\ast=(\, \n_{1_X},\n_{1_X} \,)$, where $\n_{1_X}$ is the $\N$-kernel of $1_X$.

The \emph{star-kernel} of a morphism $f:X\rightarrow Y$ is a universal star $\sigma=(\sigma_1,\sigma_2):S\rr X$ with the property $f\cdot \sigma_1=f\cdot \sigma_2$; it is easy to see that the star-kernel of $f$ coincides with $\Eq(f)^\ast \rr X$ whenever $\N$-kernels exist.

A category $\CC$ equipped with an ideal $\mathcal{N}$ of morphisms is called a \emph{multi-pointed} category \cite{GJU}. If, moreover, every morphism admits an $\N$-kernel, then $\CC$ will be called a \textit{multi-pointed category with kernels}.

\begin{definition}\cite{GJU}\label{star-regular cat}
A regular multi-pointed category $\CC$ with kernels is called a \emph{star-regular category} when every regular epimorphism in $\CC$ is a coequaliser of a star.
\end{definition}

In the total context stars are pairs of parallel morphisms, $\N$-kernels are isomorphisms, star-kernels are kernel pairs and a star-regular category is precisely a regular category. In the pointed context, the first morphism $\sigma_1$ in a star $\sigma=(\sigma_1,\sigma_2):S\rr X$ is the unique null morphism $S\rightarrow X$ and hence a star $\sigma$ can be identified with a morphism (its second component $\sigma_2$). Then, $\N$-kernels and star-kernels become the usual kernels, and a star-regular category is the same as a normal category \cite{ZJan10}, i.e.~a pointed regular category in which any regular epimorphism is a normal epimorphism.

\subsection{Calculus of star relations}\label{Calculus of star relations}
The calculus of star relations \cite{GraJanRodUrs11} can be seen as an extension of the usual calculus of relations (in a regular category) to the regular multi-pointed context.
 First of all note that for any relation $\varrho:R\rr X$ we have
$$\varrho^\ast=\varrho \Delta_X^\ast.$$

Inspired by this formula, for any relation $\varrho$ from $X$ to an object $Y$, we define
$$\varrho^\ast = \varrho \Delta_X^\ast\;\;\;\; \mathrm{and}\;\;\;\; ^\ast\!\varrho=\Delta_Y^\ast \varrho.$$

Note that associativity of composition yields
$$^\ast(\varrho^\ast)=(^\ast\!\varrho)^\ast$$
and so we can write $^\ast\!\varrho^\ast$ for the above.

For any relation $\sigma$ (from some object $Y$ to $Z$), the associativity of composition also gives
$$(\sigma^\ast) \varrho = \sigma (^\ast\!\varrho),$$
and $${(\sigma  \varrho)}^{\ast}= \sigma  {\varrho}^{\ast}.$$
% which suggests to write $$\sigma\ast\varrho$$ for the above equal composites.

It is easy to verify that for any morphism $f:X\rightarrow Y$ we have
$$f^\ast=\;^\ast f^\ast \;\;\;\;\mathrm{and}\;\;\;\; ^\ast f^{\circ}=\;^\ast {f^{\circ}}^\ast.$$

\section{$2\,$-star-permutability and star-regular pushouts}
\label{$2$-(star-)permutability}
\noindent Recall that a finitely complete category $\CC$ is called a \emph{Mal'tsev category} when any reflexive relation in $\CC$ is an equivalence relation {\cite{CLP, CKP}.
We recall the following well known characterisation of the regular categories which are Mal'tsev categories:

\begin{proposition}
\label{Definition: Mal'tsev category}
A regular category $\CC$ is a Mal'tsev category if and only if the composition of effective equivalence relations in $\CC$ is commutative:
$$\Eq(f) \Eq(g)=\Eq(g) \Eq(f)$$
for any pair of regular epimorphisms $f$ and $g$ in $\CC$ with the same domain.
\end{proposition}

There are many known characterisations of regular Mal'tsev categories (see Section $2.5$ in \cite{BB}, for instance, and references therein). The one that will play a central role in the present work is expressed in terms of commutative diagrams of the form
\begin{equation}
\label{regular po}
\vcenter{\xymatrix@=20pt{
  C \ar@{>>}[r]^-{c} \ar@<-2pt>[d]_-g & A  \ar@<-2pt>[d]_-f \\
  D \ar@{>>}[r]_-d \ar@<-2pt>[u]_-t & B, \ar@<-2pt>[u]_-s }}
\end{equation}
where $f$ and $g$ are split epimorphisms ($f\cdot s=1_B$, $g\cdot t=1_D$), $f\cdot c= d\cdot g$, $s\cdot d= c\cdot t$, and $c$ and $d$ are
regular epimorphisms. A diagram of type (\ref{regular po}) is always a pushout; it is called a \emph{regular pushout} \cite{B} (alternatively, a \emph{double extension} \cite{Jan, GRossi})
when, moreover, the canonical morphism $\langle g,c \rangle\colon C\twoheadrightarrow D\times_B A$ to the pullback $D\times_B A$ of $d$ and $f$ is a regular epimorphism.
Among regular categories, Mal'tsev categories can be characterized as those ones where any square (\ref{regular po}) is a regular pushout: this easily follows from the results in \cite{B}, and a simple proof of this fact is given in \cite{Cuboid}.

Observe that a commutative diagram of type (\ref{regular po}) is a regular pushout if and only if $c g^{\circ}= f^{\circ}  d$ or, equivalently, $g c^{\circ} = d^{\circ} f$.
This suggests to introduce the following notion:
\begin{definition}\label{Definition srp}
A commutative diagram $(\ref{regular po})$ is a \emph{star-regular pushout} if it satisfies the identity ${cg^{\circ}}^\ast = f^{\circ} d^\ast$ (or, equivalently, $g{c^{\circ}}^* = d ^{\circ}f^*$).
\end{definition}

Diagrammatically, the property of being a star-regular pushout can be expressed as follows. Consider the commutative diagram
\begin{equation}
\label{star-regular po explaination}
\vcenter{\xymatrix@C=10pt@R=5pt{
    \Nk_g \ar@{>>}[dr] \ar@{ >->}[ddd]_-{\n_g}  & & & & &  \\
    & \Nk_a \ar@{ >->}[ddd]^(.4){\n_a} \ar@{.>}[r] & \Nk_x \ar@{ >->}[dddd]^(.3){\n_x} \\ \\
    C \ar@{>>}[dr]_-p \ar@<-2pt>[ddddd]_-g \ar@{>>}[rrrrr]^(.7)c & & & & & A \ar@<-2pt>[ddddd]_-f \\
    & M \ar[ddddl]_(.3)a \ar[rrrru]^(.3)b \ar@{ >->}[dr]_-{m} \\
    & & D\times_B A \ar@{}[dddrrr]|(.25){\lrcorner} \ar@<-2pt>[dddll]_-x \ar@{>>}[rrruu]^-y \\ \\ \\
    D \ar@<-2pt>[uuuuu]\ar@{>>}[rrrrr]_-d & & & & & B, \ar@<-2pt>[uuuuu]}}
\end{equation}
where $(D\times_B A, x,y)$ is the pullback of $(f,d)$, $m\cdot p$ is the (regular epimorphism, monomorphism) factorisation of the induced morphism $\langle g, c \rangle \colon C \rightarrow D \times_B A$.
% $a=x\cdot m$ %and $b=y\cdot m$.
Then the identity $c g^{\circ}=b a^{\circ}$ allows one to identify $c {g^{\circ}}^\ast$ with the relation $(a \cdot \n_a, b\cdot \n_a)$, while $f^{\circ}d=y x^{\circ}$ says that $f^{\circ} d^\ast$ can be identified with the relation $(x \cdot \n_x, y \cdot \n_x)$. Accordingly, diagram (\ref{regular po}) is a star-regular pushout precisely when the dotted arrow from $\Nk_a$ to $\Nk_x$ is an isomorphism. Notice that in the total context the $\N$-kernels are isomorphisms, so that $m$ is an isomorphism if and only if (\ref{regular po}) is a regular pushout, as expected.

The ``star-version'' of the notion of Mal'tsev category can be defined as follows:
\begin{definition}\cite{GraJanRodUrs11} \label{2-star-permutable}
A regular multi-pointed category with kernels $\mathbb{C}$ is said to be a \emph{$2$-star-permutable category} if
$$\Eq(f) \Eq(g)^\ast=\Eq(g) \Eq(f)^\ast$$
for any pair of regular epimorphisms $f$ and $g$ in $\CC$ with the same domain.
\end{definition}
One can check that the equality $\Eq(f) \Eq(g)^\ast=\Eq(g) \Eq(f)^\ast$ in the definition above can be actually replaced by $\Eq(f) \Eq(g)^\ast \le\Eq(g) \Eq(f)^\ast$.

In the total context the property of $2$-star-permutability characterises the regular categories which are Mal'tsev. In the pointed context this same property characterises the regular categories which are subtractive \cite{ZJ1} (this follows from the characterisation of subtractivity given in Theorem 6.9 in \cite{ZJan06}).
%\begin{remark}
%\label{calculus}
%Observe that The fact that the vertical morphisms $g$ and $f$ are split epimorphisms, implies that there is a split epimorphism from $\Eq(c)$ to $\Eq(d)$. Consequently, the image of $\Eq(c)$ along $g$ is $\Eq(d)$, which %can be written as:
%$$g\langle \Eq(c) \rangle=\Eq(d)$$ i.e. $$gc^{\circ}cg^{\circ}=d^{\circ}d.$$
%\end{remark}

The next result gives a useful characterisation of $2$-star-permutable categories. Given
a commutative diagram of type (\ref{regular po}), we write $g\langle \Eq(c) \rangle$ and $g\langle \Eq(c)^\ast \rangle$ for the direct images of the relations $\Eq(c)$ and $\Eq(c)^\ast$ along the split epimorphism $g$.
The vertical split epimorphisms are such that both the equalities $g\langle \Eq(c) \rangle=\Eq(d)$ and $g\langle \Eq(c)^\ast \rangle=\Eq(d)^\ast$ hold true in $\mathbb C$.

\begin{proposition}
\label{starregularpushoutscharacterisation}
For a regular multi-pointed category with kernels $\CC$ the following statements are equivalent:
\begin{enumerate}
    \item[(a)] $\CC$ is a $2$-star-permutable category;
    \item[(b)] any commutative diagram of the form \emph{(\ref{regular po})} is a star-regular pushout.
\end{enumerate}
\end{proposition}
\begin{proof}
\noindent (a) $\Rightarrow$ (b)
Given a pushout (\ref{regular po}) we have
$$\begin{tabular}{lll}
    $f^{\circ} d^\ast$ & $= cc^{\circ}f^{\circ}d^\ast$ &  (Lemma~\ref{pps of ms as relations}(2)) \\
    & $= cg^{\circ}d^{\circ}d^\ast$ &  ($f\cdot c=d \cdot g$) \\
    & $= cg^{\circ} g c^{\circ}c^\ast g^{\circ}$ & ($\Eq(d)^\ast=g\langle \Eq(c)^\ast \rangle$) \\
    & $= cc^{\circ}cg^{\circ}g^\ast g^{\circ}$ & ($\Eq(g) \Eq(c)^\ast=\Eq(c) \Eq(g)^\ast$ by Definition~\ref{2-star-permutable}) \\
     & $\le  cc^{\circ}cg^{\circ}g g^{\circ}$ & ($g^* \le g$) \\
    & $= cg^{\circ} $. &  (Lemma~\ref{pps of ms as relations}(1)) \\
  %  & $=cg^{\circ}\Delta_D^\ast$ & (Lemma~\ref{saturating regular epis}) \\
   % & $={cg^{\circ}}$. & (Subsection~\ref{Calculus of star relations})
\end{tabular}$$
Since ${cg^{\circ}}^*$ is the largest star contained in $cg^{\circ}$, it follows that $f^{\circ} d^\ast \le {cg^{\circ}}^*$. The inclusion ${cg^{\circ}}^* \le f^{\circ} d^\ast$ always holds, so that ${cg^{\circ}}^* = f^{\circ} d^\ast$.

\noindent (b) $\Rightarrow$ (a)
 Let us consider regular epimorphisms $f\colon X \twoheadrightarrow Y$ and $g\colon X\twoheadrightarrow Z$. We want to prove that $\Eq(f) \Eq(g)^\ast = \Eq(g) \Eq(f)^\ast$. For this we build the following diagram
$$
\xymatrix@=20pt{
    \Eq(f) \ar@<-3pt>[d]_-{\pi_1} \ar@<3pt>[d]^-{\pi_2} \ar@{>>}[r]^-{c} & g\langle \Eq(f) \rangle \ar@<-3pt>[d]_-{\rho_1} \ar@<3pt>[d]^-{\rho_2} \\
    X \ar@{>>}[d]_-f \ar@{>>}[r]_-g & Z\\ Y }
$$
that represents the regular image of $\Eq(f)$ along $g$. The relation $g \langle \Eq(f) \rangle=(\rho_1,\rho_2)$ is reflexive and, consequently, $\rho_1$ is a split epimorphism. By assumption, we then know that the equality
$$(A) \qquad  \rho_1^{\circ} g^* = c {\pi_1^{\circ}}^*$$
holds true. This implies that
$$\begin{tabular}{lll}
$\Eq(f) \Eq(g)^*$   & =   $\pi_2 \pi_1^{\circ} g^{\circ} g^*$ &  \\
  & $=\pi_2 c^{\circ} \rho_1^{\circ} g^*$   & ($g \cdot \pi_1 =\rho_1 \cdot c$)\\
    & $= \pi_2 c^{\circ} c {\pi_1^{\circ}} ^*$  &  (A) \\
    & $\le  \pi_2 c^{\circ} c \pi_2^{\circ} \pi_2 {\pi_1^{\circ}} ^*$  & ($ \Delta_{\Eq(f)} \le \pi_2^{\circ} \pi_2$) \\
    & $=  \Eq(g) \pi_2 {\pi_1^{\circ}} ^*$ & ($\pi_2\langle \Eq(c)\rangle =\Eq(g
    ))$ \\
    & $=  \Eq(g) \Eq(f)^*,$ &  \\
\end{tabular}$$
where the equality $\pi_2\langle \Eq(c) \rangle =\Eq(g)$ follows from the fact that the split epimorphisms $\pi_2$ and $\rho_2$ induce a split epimorphism from $\Eq(c)$ to $\Eq(g)$.
\end{proof}

In the total context, Proposition~\ref{starregularpushoutscharacterisation} gives the characterisation of regular Mal'tsev categories through regular pushouts (see~\cite{B} and Proposition 3.4 of~\cite{Cuboid}), as expected. In the pointed context, condition (b) of Proposition~\ref{starregularpushoutscharacterisation} translates into the pointed version of the \emph{right saturation} property \cite{{GraJanRodUrs11}} for any commutative diagram of type (\ref{regular po}): the induced morphism $\bar{c}:\Ker(g) \rightarrow \Ker(f)$, from the kernel of $g$ to the kernel of $f$ is also a regular epimorphism. This can be seen by looking at diagram (\ref{star-regular po explaination}), where the $\N$-kernels now represent actual kernels, so that $\Ker(a)=\Ker(x)=\Ker(f)$.

%\begin{proposition}
%\label{Mal'tsev characterisations} \emph{\cite{B, Cuboid}}
%For a regular category $\CC$ the following statements are equivalent:
%\begin{enumerate}
%    \item[(a)] $\CC$ is a Mal'tsev category;
 %   \item[(b)] any commutative diagram of the form \emph{(\ref{regular po})} is a regular pushout; equivalently $f^{\circ} d = cg^{\circ}$;
%    \item[(c)] for any a commutative diagram
%\begin{equation}\label{cube}
%\xymatrix@C=25pt@R=15pt{
%    W \times_D C \ar@<-2pt>[dd]_-k \ar[dr]^-{\gamma} \ar@{.>}[rr]^-v & & Y \times_B A \ar@<-2pt>@{-->}[dd]_(.3){h} \ar[dr]^-{\alpha} \\
 %   & C \ar@<-2pt>[dd]_(.7){g} \ar@{>>}[rr]^(.2){c} & & A \ar@<-2pt>[dd]_-{f} \\
 %   W \ar@<-2pt>[uu]_-j \ar[dr]_-{\delta} \ar@{-->>}[rr]^(.7){w} & & Y \ar@<-2pt>@{-->}[uu]_(.7)i \ar@{-->}[dr]_(.4){\beta} \\
  %  & D \ar@<-2pt>[uu]_(.3)t \ar@{>>}[rr]_-d & & B, \ar@<-2pt>[uu]_-s }
% \end{equation}
% where the front square is of the form \emph{(\ref{regular po})}, $\beta \cdot w = d \cdot \delta$,
% $w$ is a regular epimorphism, $(W \times_D C, k, \gamma)$ and $(Y \times_B A, h, \alpha)$ are pullbacks,
%then the comparison morphism $v \colon W \times_D C \rightarrow Y \times_B A$ is also a regular epimorphism.
%\end{enumerate}
%\end{proposition}

\subsection{The star of a pullback relation}\label{The star of a pullback relation}
Consider the pullback relation $\pi=(\pi_1,\pi_2)$ of a pair $(g,\delta)$ of morphisms as in the diagram
$$
\xymatrix@C=6pt@R=8pt{
    & W\times_D C \ar[dd]_-{\pi_1} \ar[rr]^-{\pi_2} \ar@{}[ddrr]|(.3){\lrcorner} & & C \ar[dd]^-g\\ \\
    & W \ar[rr]_-{\delta} & & D.}
$$
The \emph{star of the pullback relation} $\pi$ is defined as $\pi^\ast =\pi \Delta_W^\ast$. It can be described as the universal relation $\nu=(\nu_1,\nu_2)$ from $W$ to $C$ such that $\nu_1\in\N$ and $\delta\cdot \nu_1=g\cdot \nu_2$ as in the diagram
$$
\xymatrix@C=6pt@R=8pt{
(W\times_D C)^\ast \ar@(d,l)[dddr]_-{\nu_1} \ar@(r,u)[rrrd]^-{\nu_2} \ar[dr]^(.6){\n_{\pi_1}}\\
    & W\times_D C \ar[dd]_-{\pi_1} \ar[rr]^-{\pi_2} \ar@{}[ddrr]|(.3){\lrcorner} & & C \ar[dd]^-g\\ \\
    & W \ar[rr]_-{\delta} & & D,}
$$
where $\n_{\pi_1}$ is the $\N$-kernel of $\pi_1$, $\nu_1=\pi_1 \cdot \n_{\pi_1}$ and $\nu_2=\pi_2 \cdot \n_{\pi_1}$.

By using the composition of relations one has the equalities $\pi=\pi_2 \pi_1^{\circ}=g^{\circ}\delta$, so that
$$\pi^*={\pi_2 \pi_1^{\circ}}^\ast={g^{\circ}\delta}^\ast.$$

In the total context, the star of a pullback relation is precisely that pullback relation. In the pointed context, the star of the pullback (relation) of $(g,\delta)$ is given by $\pi^\ast=(0,\ker(g))$.

A morphism $f\colon X\rightarrow Y$ in a multi-pointed category with kernels is said to be \emph{saturating} \cite{GraJanRodUrs11} when the induced dotted morphism from the $\N$-kernel of $1_X$ to the $\N$-kernel of $1_Y$ making the diagram
$$
\xymatrix@=20pt{ N_{1_X} \ar@{ >->}[d]_-{\n_{1_X}} \ar@{.>>}[r] & N_{1_Y} \ar@{ >->}[d]^-{\n_{1_Y}} \\
    X \ar[r]_-f & Y}
$$
commute is a regular epimorphism. All morphisms are saturating in the pointed context.
This is also the case for any \emph{quasi-pointed category} \cite{Bou}, namely a finitely complete category with an initial object $0$ and a terminal object $1$ such that the arrow $0 \rightarrow 1$ is a monomorphism. As in the pointed case, it suffices to choose for $\N$ the class of morphisms which factor through the initial object $0$. In this case we shall speak of the  \emph{quasi-pointed context}.
In the total context, any regular epimorphism is saturating.
The proof of the following result is straightforward:

\begin{lemma}\emph{\cite{GraJanRodUrs11}}\label{saturating regular epis}
Let $\CC$ be a regular multi-pointed category with kernels. For a morphism $f:X\rightarrow Y$ the following conditions are equivalent:
\begin{enumerate}
    \item[(a)] $f$ is saturating;
    \item[(b)] $\Delta_Y^\ast=f^\ast f^{\circ}$.
\end{enumerate}
\end{lemma}

%We can extend this result to arbitrary relations and the following equivalent conditions also hold both in the total and pointed contexts.

%\begin{proposition}\label{Prop: enough trivial onjs 2}
%Let $\mathbb{C}$ be a regular multi-pointed category with kernels. The following conditions are equivalent:
%\begin{enumerate}
%    \item[(a)] If a relation $\langle \varrho_1,\varrho_2\rangle:R\rightarrowtail X\times Y$ is such that $\varrho_1,\varrho_2\in\N$, then $R$ is an $\N$-trivial object;
%    \item[(b)] If $\langle \sigma_1,\sigma_2\rangle: S\rightarrowtail X\times Y$ is a relation such that $\sigma_1\cdot n,\sigma_2\cdot n \in \mathcal{N}$, then $n\in\mathcal{N}$;
%    \item[(c)] $\mathbb{C}$ has enough trivial objects and $\N$-trivial objects are closed under products.
%\end{enumerate}
%\end{proposition}
%\proof The proof is similar to that of Proposition~\ref{Prop: enough trivial onjs}.
%\endproof

The next result gives a characterisation of $2$-star-permutable categories which will be useful in the following section.

\begin{proposition}
\label{star-Mal'tsev characterisations}
For a regular multi-pointed category  $\CC$ with kernels and saturating regular epimorphisms the following statements are equivalent:
\begin{enumerate}
    \item[(a)] $\CC$ is a $2$-star-permutable category;
 %   \item[(b)] for any commutative diagram of the form \emph{(\ref{regular po})}, we have $f^{\circ} d^\ast = {cg^{\circ}}^\ast$;
    \item[(b)] for any commutative diagram
\begin{equation}\label{star-cube}
%\xymatrix@C=25pt@R=15pt{
\vcenter{\xymatrix@!0@C=60pt@R=33pt{
    (W\times_D C)^\ast \ar[dd]_-{\nu_1} \ar[dr]^-{\nu_2} \ar@{.>}[rr]^-{\lambda} & & (Y\times_B A)^\ast \ar@{-->}[dd]_(.3){\chi_1} \ar[dr]^-{\chi_2} \\
    & C \ar@<-2pt>[dd]_(.7){g} \ar@{>>}[rr]^(.2){c} & & A \ar@<-2pt>[dd]_-{f} \\
    W \ar[dr]_-{\delta} \ar@{-->>}[rr]^(.7){w} & & Y \ar@{-->}[dr]_(.4){\beta} \\
    & D \ar@<-2pt>[uu]_(.3)t \ar@{>>}[rr]_-d & & B, \ar@<-2pt>[uu]_-s }}
\end{equation}
where the front square is of the form \emph{(\ref{regular po})}, $\beta \cdot w = d \cdot \delta$,
$w$ is a regular epimorphism, $((W\times_D C)^\ast, \nu_1, \nu_2)$ and $((Y\times_B A)^\ast, \chi_1, \chi_2)$ are stars of the corresponding pullback relations, then the comparison morphism $\lambda \colon (W\times_D C)^\ast \rightarrow (Y\times_B A)^\ast$ is also a regular epimorphism.
\end{enumerate}
\end{proposition}
\proof
\noindent (a) $\Rightarrow$ (b)
 To prove that the arrow $\lambda$ in the cube above is a regular epimorphism, we must show that $\langle \chi_1, \chi_2 \rangle \lambda$ in the commutative diagram
 $$
\xymatrix@C=20pt{(W\times_D C)^\ast \ar[r]^-{\lambda} \ar@{ >->}[d]_-{\langle \nu_1, \nu_2 \rangle} &
    (Y\times_B A)^\ast \ar@{ >->}[d]^-{\langle \chi_1, \chi_2 \rangle} \\
    W\times C \ar@{>>}[r]_-{w\times c} & Y\times A}
$$
is the (regular epimorphism, monomorphism) factorisation of the morphism $\langle w\cdot \nu_1, c\cdot \nu_2 \rangle \colon (W\times_D C)^\ast \rightarrow Y\times A$. That is, we must have $c\nu_2\nu_1^{\circ}w^{\circ}=\chi_2\chi_1^{\circ}$ or, equivalently, $cg^{\circ}\delta^\ast w^{\circ}=f^{\circ}\beta^\ast$, since $\nu_2\nu_1^{\circ}=\nu^\ast=g^{\circ}\delta^\ast$ and $\chi_2\chi_1^{\circ}=\chi^\ast=f^{\circ}\beta^\ast$ (see Section~\ref{The star of a pullback relation}).

The front square of diagram (\ref{star-cube}) is a star-regular pushout by Proposition \ref{starregularpushoutscharacterisation}, which means that the equality
$$
(B) \qquad c{g^{\circ }}^*= f^{\circ }d^*
$$
holds true. Now,
we always have
$$ \begin{tabular}{lll}
    $cg^{\circ}\delta^\ast w^{\circ}$ & $\leqslant f^{\circ}d \delta^\ast w^{\circ}$ & (commutativity of the front face of (\ref{star-cube})) \\
    & $= f^{\circ}\beta w^\ast w^{\circ}$ & ($d\cdot \delta=\beta\cdot w$) \\
    & $= f^{\circ}\beta \Delta_Y^\ast$ & (Lemma~\ref{saturating regular epis}) \\
    & $= f^{\circ}\beta^\ast$.
\end{tabular}
$$
The other inequality follows from
$$\begin{tabular}{lll}
    $cg^{\circ}\delta^\ast w^{\circ}$ & $\geqslant c{g^{\circ}}^\ast \delta^\ast w^{\circ}$ &  ($g^{\circ} \geqslant {g^{\circ}}^{\ast}$) \\
    & $= f^{\circ}d^\ast \delta^\ast w^{\circ}$ & (B) \\
    & $= f^{\circ}d \delta^\ast w^{\circ}$ & ($^\ast \delta^\ast= \delta^\ast$; Section~\ref{Calculus of star relations}) \\
    & $= f^{\circ}\beta^\ast$. & (as in the inequality above)
\end{tabular}$$

\noindent (b) $\Rightarrow$ (a) A commutative diagram of type (\ref{regular po}) induces a commutative cube
$$
\xymatrix@!0@C=60pt@R=33pt{
    \Nk_g \ar[dd]_-{g\cdot \n_g} \ar[dr]^-{\n_g} \ar@{.>>}[rr]^-{\lambda} & & (D \times_B A)^* \ar@{-->}[dd]_(.3){\chi_1}  \ar[dr]^-{\chi_2} \\
    & C \ar@<-2pt>[dd]_(.7){g} \ar@{>>}[rr]^(.2){c} & & A \ar@<-2pt>[dd]_-{f} \\
    D \ar@{=}[dr] \ar@{=}[rr] & & D  \ar[dr]_(.4){d} \\
    & D \ar@<-2pt>[uu]_(.3)t \ar@{>>}[rr]_-d & & B, \ar@<-2pt>[uu]_-s }
$$
where $\nu=(g\cdot \n_g,\n_g)$ is the star of the pullback (relation) of $(g,1_D)$. By assumption, $\lambda$ is a regular epimorphism which translates into the equality $cg^{\circ}1_D^\ast 1_D={f^{\circ}d}^\ast$, as observed in the first part of the proof. We get the equality ${cg^{\circ}}^\ast=f^{\circ} d^\ast$,  and this proves that diagram (\ref{regular po}) is a star-regular pushout and, consequently, that $\CC$ is a $2$-star-permutable category by Proposition~\ref{starregularpushoutscharacterisation}.

\endproof

In the total context, Proposition~\ref{star-Mal'tsev characterisations} is the ``star version'' of Proposition $3.6$ in \cite{Cuboid} (see also Proposition $4.1$ in \cite{B}). In the pointed context condition (b) of Proposition~\ref{star-Mal'tsev characterisations} also reduces to the pointed version of the right saturation property (in the sense of \cite{GraJanRodUrs11}).  Indeed, in this context that condition says that, in the following commutative diagram
\begin{equation}\label{pointed context cube}
%\xymatrix@C=25pt@R=15pt{
\vcenter{\xymatrix@!0@C=60pt@R=33pt{
    \Ker(g) \ar[dd]_-{0} \ar[dr]^-{\ker(g)} \ar@{.>}[rr]^-{\bar{c}} & & \Ker(f) \ar@{-->}[dd]_(.3){0} \ar[dr]^-{\ker(f)} \\
    & C \ar@<-2pt>[dd]_(.7){g} \ar@{>>}[rr]^(.2){c} & & A \ar@<-2pt>[dd]_-{f} \\
    W \ar[dr]_-{\delta} \ar@{-->>}[rr]^(.7){w} & & Y  \ar@{-->}[dr]_(.4){\beta} \\
    & D \ar@<-2pt>[uu]_(.3)t \ar@{>>}[rr]_-d & & B, \ar@<-2pt>[uu]_-s }}
\end{equation}
the induced arrow $\overline{c} \colon \Ker(g) \rightarrow \Ker(f)$ is a regular epimorphism.

We conclude this section with the pointed version of Propositions~\ref{starregularpushoutscharacterisation} and ~\ref{star-Mal'tsev characterisations}:

\begin{corollary}\emph{(see Theorem $2.12$ in \cite{GraJanRodUrs11})}
\label{subtractive characterisations}
For a pointed regular category $\CC$ the following statements are equivalent:
\begin{enumerate}
    \item[(a)] $\CC$ is a subtractive category;
    \item[(b)] any commutative diagram of the form \emph{(\ref{regular po})} is right saturated, i.e. the comparison morphism $\bar{c}\colon \Ker(g) \rightarrow \Ker(f)$ is a regular epimorphism.
\end{enumerate}
\end{corollary}

%%%%%%%%%%%%%%%%%%%%%%%%%%%%%%%%%%%%%%%  SECTION: The (Star-)Cuboid Lemma  %%%%%%%%%%%%%%%%%%%%%%%%%%%%%%%%%%%%%%%%%%%%%%%%%%%%%%
\section{The Star-Cuboid Lemma}
\label{The (Star-)Cuboid Lemma}
In \cite{Cuboid} it was shown that regular Mal'tsev categories can be characterised through the validity of a homological lemma called the Upper Cuboid Lemma, a strong form of the denormalised $3 \times 3$ Lemma \cite{B,L,GR}. We are now going to extend this result to the star-regular context. We shall then observe that, in the pointed context, it gives back the classical Upper $3 \times 3$ Lemma characterising subtractive normal categories.

\subsection{$\N$-trivial objects}
An object $X$ in a multi-pointed category is said to be \textit{$\N$-trivial} when $1_X\in \N$. If a composite $f\cdot g$ belongs to $\N$ and $g$ is a strong epimorphism, then also $f$ belongs to $\N$. This implies that $\N$-trivial objects are closed under strong quotients. One says that a multi-pointed category
 $\mathbb C$ \emph{has enough trivial objects} \cite{Star3x3} when $\N$ is a closed ideal \cite{Grandis92}, i.e.~any morphism in $\N$ factors through an $\N$-trivial object and, moreover, the class of $\N$-trivial objects is closed under subobjects and squares, where the latter property means that, for any $\N$-trivial object $X$, the object $X^2=X\times X$ is $\N$-trivial.
An equivalent way of expressing the existence of enough trivial objects is recalled in the following:
\begin{proposition}\emph{\cite{Star3x3}}\label{Prop: enough trivial onjs}
Let $\mathbb{C}$ be a regular multi-pointed category with kernels. The following conditions are equivalent:
\begin{enumerate}
    \item[(a)] if $(\sigma_1,\sigma_2): S\rightrightarrows X$ is a relation on $X$ such that $\sigma_1\cdot n \in \mathcal{N} $ and $\sigma_2\cdot n \in \mathcal{N}$, then $n\in\mathcal{N}$;
    \item[(b)] $\mathbb{C}$ has enough trivial objects.
\end{enumerate}
\end{proposition}
In the following we shall also assume that $\N$-trivial objects are closed under binary products. Remark that in the total and in the (quasi-)pointed contexts there are enough trivial objects, and $\N$-trivial objects are closed under binary products.

Under the presence of enough trivial objects the assumption that  $\N$-trivial objects are closed under binary products is equivalent to the following condition:
\begin{enumerate}
\item[(a')]  if $(\sigma_1,\sigma_2): S\rightarrowtail X \times Y$ is a relation from $X$ to $Y$ such that $\sigma_1\cdot n \in \mathcal{N} $ and $\sigma_2\cdot n \in \mathcal{N}$, then $n\in\mathcal{N}$.
\end{enumerate}
Whenever the category has enough trivial objects, condition (a') implies that star-kernels ``commute'' with stars of pullback relations:

\begin{lemma}\label{starlimits}
Let $\mathbb C$ be a multi-pointed category with kernels, enough trivial objects, and assume that $\N$-trivial objects are closed under binary products. Given a commutative cube
$$
\vcenter{
%\xymatrix@C=25pt@R=15pt{
\xymatrix@!0@C=60pt@R=33pt{
    (W \times_D C)^\ast \ar[dd]_-{\nu_1} \ar[dr]^-{\nu_2} \ar[rr]^-{\lambda} & & (Y\times_B A)^\ast \ar@{-->}[dd]_(.3){\chi_1} \ar[dr]^-{\chi_2} \\
    & C \ar[dd]_(.7){g} \ar@{>>}[rr]^(.2){c} & & A \ar[dd]_-{f} \\
    W \ar[dr]_-{\delta} \ar@{-->>}[rr]^(.7){w} & & Y \ar@{-->}[dr]_(.4){\beta} \\
    & D  \ar@{>>}[rr]_-d & & B}}
$$
in $\mathbb C$, consider the star-kernels of $c$, $d$ and $w$, and the induced morphisms $\overline{\delta} \colon \Eq(w)^* \rightarrow \Eq(d)^*$ and
$\overline{g} \colon \Eq(c)^* \rightarrow \Eq(d)^*$. Then the following constructions are equivalent (up to isomorphism):

\begin{itemize}
\item taking the horizontal star-kernel of $\lambda$ and then the induced morphisms $\Eq(\lambda)^* \rightarrow \Eq(w)^*$ and $\Eq(\lambda)^* \rightarrow \Eq(c)^*$;
\item taking the star of the pullback (relation) of $\overline{g}$ and $\overline{\delta}$ and then the induced morphisms $(\Eq(w)^\ast\times_{\Eq(d)^*} \Eq(c)^\ast)^\ast \rightrightarrows (W \times_D C)^*$. \end{itemize}
\end{lemma}
\proof
This follows easily by the usual commutation of kernel pairs with pullbacks and condition (a').
\endproof

\label{suf conds}
%Note that the commutativity of the diagram implies that $d$ is a regular epimorphism, $d\cdot s_1=d\cdot s_2$ since $\bar{g}$ is an epimorphism and, moreover, that $v\cdot t_1=v\cdot t_2$, because the pair of morphisms %$(h,\alpha)$ is jointly monomorphic. Finally, $\Eq(w)\times_S \Eq(c)=\Eq(v)$ when $S=\Eq(d)$ since ``kernel pairs commute with pullbacks''. Consequently, to prove the exactness of the upper row, it suffices to show that %%$v$ is a regular epimorphism.

In a star-regular category, a \textit{(short) star-exact sequence} is a diagram
$$
\xymatrix{ \Eq(f)^\ast \ar@<3pt>[r]^-{f_1} \ar@<-3pt>[r]_-{f_2} & X\ar@{>>}[r]^-{f} & Y}
$$
where $\Eq(f)^*$ is a star-kernel of $f$ and $f$ is a coequaliser of ${f_1}$ and ${f_2}$ (which, by star-regularity, is the same as to say that $f$ is a regular epimorphism). In the total context, a star-exact sequence is just an exact fork, while in the (quasi-)pointed context it is a short exact sequence in the usual sense.
\vspace{3mm}

\noindent\textbf{The Star-Upper Cuboid Lemma} \\
Let $\CC$ be a star-regular category. Consider a commutative diagram of morphisms and stars in $\mathbb C$
\begin{equation}
\label{star-regular 4x3}
\vcenter{\xymatrix@!0@C=24pt@R=40pt{
    & P \ar@{>>}[dl]_(.6){\tau_1} \sarl{\pi}{rrrrr} \ar[ddr]^(.75){\tau_2} & & & & &
    (W\times_D C)^\ast \ar@{-->>}[dl]_(.6){\nu_1} \ar[rrrrr]^-{\lambda} \ar[ddr]^(.75){\nu_2} & & & & &
    (Y\times_B A)^\ast \ar@{-->>}[dl]_(.6){\chi_1} \ar[ddr]^(.75){\chi_2} \\
    \Eq(w)^\ast \sarldash{}{rrrrr} \ar[ddr]_-(.2){\bar{\delta}} & & & & &
    W \ar@{-->>}[rrrrr]^(.73){w} \ar@{-->}[ddr]_-(.2){\delta} & & & & &
    Y \ar@{-->}[ddr]_-(.2){\beta} \\
    & & \Eq(c)^\ast \ar@{>>}[dl]_(.6){\bar{g}} \sarl{}{rrrrr} & & & & &
    C \ar@{>>}[dl]_(.6){g} \ar@{>>}[rrrrr]_(.27){c} & & & & & A \ar@{>>}[dl]_(.6){f}  \\
    & S \sarr{\sigma}{rrrrr} & & & & &
    D \ar[rrrrr]_-{d} & & & & & B, }}
\end{equation}
where the three diamonds are stars of pullback (relations) of regular epimorphisms along arbitrary morphisms (so that $P=(\Eq(w)^\ast\times_S \Eq(c)^\ast)^\ast$) and the two middle rows are star-exact sequences. Then \emph{the upper row is a star-exact sequence whenever the lower row is}.

Note that, in the diagram \eqref{star-regular 4x3} above, $d$ is necessarily a regular epimorphism, $d\cdot \sigma_1=d\cdot \sigma_2$ since $\bar{g}$ is an epimorphism, and $\lambda\cdot \pi_1=\lambda\cdot \pi_2$, because the pair of morphisms $(\chi_1,\chi_2)$ is jointly monomorphic.

\begin{theorem}\label{thm star-upper cuboid}
Let $\CC$ be a star-regular category with saturating regular epimorphisms, enough trivial objects, and assume that $\N$-trivial objects are closed under binary products. The following conditions are equivalent:
\begin{enumerate}
\item[(a)] $\CC$ is a $2$-star-permutable category;
\item[(b)] the Star-Upper Cuboid Lemma holds true in $\CC$.
\end{enumerate}
\end{theorem}
\proof
\noindent (a) $\Rightarrow$ (b)
 Suppose that the lower row is a star-exact sequence. The fact that $\pi=\Eq({\lambda})^\ast$ follows from Lemma \ref{starlimits}.
As explained in Proposition~\ref{star-Mal'tsev characterisations}, $\lambda$ is a regular epimorphism if and only if $c g^{\circ}\delta^\ast w^{\circ}\geqslant f^{\circ}\beta^\ast$. In fact we have
$$\begin{tabular}{lll}
    $cg^{\circ}\delta^\ast w^{\circ}$ & $=cc^{\circ}cg^{\circ}gg^{\circ}\delta^\ast w^{\circ}$ &  (Lemma~\ref{pps of ms as relations}(1)) \\
    & $\geqslant cc^{\circ}cg^{\circ}g^\ast g^{\circ}\delta^\ast w^{\circ}$ &  ($\Eq(g) \geqslant \Eq(g)^\ast$) \\
    & $= cg^{\circ}gc^{\circ}c^\ast g^{\circ}\delta^\ast w^{\circ}$ & ($\Eq(c)\Eq(g)^\ast=\Eq(g)\Eq(c)^\ast$; Definition~\ref{2-star-permutable}) \\
    & $= cg^{\circ}d^{\circ}d^\ast\delta^\ast w^{\circ}$ & ($g\langle \Eq(c)^\ast \rangle=\Eq(d)^\ast$ by assumption) \\
    & $= cg^{\circ}d^{\circ}d\delta^\ast w^{\circ}$ & ($^\ast \delta^\ast= \delta^\ast$; Section~\ref{Calculus of star relations}) \\
    & $= cc^{\circ}f^{\circ}\beta w^\ast w^{\circ}$ & ($d\cdot g=f\cdot c$, $d\cdot \delta=\beta\cdot w$) \\
    & $= f^{\circ}\beta w^\ast w^{\circ}$ & (Lemma~\ref{pps of ms as relations}(2)) \\
    & $= f^{\circ}\beta \Delta_Y^\ast$ & (Lemma~\ref{saturating regular epis}) \\
    & $= f^{\circ}\beta^\ast$. & (Section~\ref{Calculus of star relations})
\end{tabular}$$

\noindent (b) $\Rightarrow$ (a) Consider a commutative cube of the form (\ref{star-cube}). We construct a commutative diagram of type (\ref{star-regular 4x3}) by taking the star-kernels of $c$, $w$, $d$ and $\lambda$, so that $\bar{g}$, $\bar{\delta}$, $\tau_1$ and $\tau_2$ are the induced arrows between the star-kernels. By Lemma \ref{starlimits} we know that $(\tau_1,\tau_2)$ is the star above the pullback (relation) of $(\bar{g}, \bar{\delta})$. By applying the Star-Upper Cuboid Lemma to this diagram we conclude that the upper row is a star-exact sequence and, consequently, $\lambda$ is a regular epimorphism. By Proposition~\ref{star-Mal'tsev characterisations}, $\CC$ is a $2$-star-permutable category.
\endproof

In the total context, Theorem~\ref{thm star-upper cuboid} is precisely Theorem $4.3$ in \cite{Cuboid}, which gives a characterisation of regular Mal'tsev categories through the Upper Cuboid Lemma, as expected. In the pointed context, the Star-Upper Cuboid Lemma gives the classical Upper $3 \times 3$ Lemma: in the pointed version of diagram (\ref{star-regular 4x3}), the back part is irrelevant (like in diagram (\ref{pointed context cube})). Then the front part is a $3\times 3$ diagram where all columns and the middle row are short exact sequences. The Star-Upper Cuboid Lemma claims that the upper row is a short exact sequence whenever the lower row is, i.e. the same as the Upper $3\times 3$ Lemma. The pointed version of Theorem~\ref{thm star-upper cuboid} is Theorem 5.4 of \cite{ZJan10} which characterises normal subtractive categories.
Note that in the pointed context, the Upper $3 \times 3$ Lemma is also equivalent to the Lower $3 \times 3$ Lemma as shown in \cite{ZJan10}.

\end{document}